\documentclass[12pt]{amsart}
\usepackage{amsmath, amssymb, amsthm}
\usepackage{paralist, xcolor, tikz, hyperref}

\usepackage[margin=1in,letterpaper,portrait]{geometry}

\usepackage{caption}
\usepackage{enumitem}
\usepackage[noadjust]{cite}

\newcommand{\RR}{\mathbb{R}}
\newcommand{\NN}{\mathbb{N}}
\newcommand{\ZZ}{\mathbb{Z}}

\newcommand{\affS}{\widetilde{S}}
\newcommand{\wt}{\widetilde}
\newcommand{\ol}{\overline}

\newcommand{\height}{\op{ht}}
\newcommand{\dis}{\operatorname{dis}}

\newcommand{\lS}{\ell_S}

\newcommand{\op}{\operatorname}

\newtheorem{conjecture}{Conjecture}[section]
\newtheorem{theorem}[conjecture]{Theorem}

\newtheorem{lemma}[conjecture]{Lemma}
\newtheorem{proposition}[conjecture]{Proposition}

\theoremstyle{definition}
\newtheorem{question}[conjecture]{Question}

\newtheorem{remark}[conjecture]{Remark}
\newtheorem{definition}[conjecture]{Definition}
\newtheorem*{question*}{Question}

\newcommand{\negphantom}[1]{\settowidth{\dimen0}{#1}\hspace*{-\dimen0}}

\title{Disarray, reduced words, and $321$-avoidance in George groups}

\author[Joel Brewster Lewis]{Joel Brewster Lewis$^*$}
\address{Department of Mathematics, George Washington University, Washington, DC, USA}
\email{jblewis@gwu.edu}
\thanks{$^*$Research partially supported by a grant from the Simons Foundation (634530).}

\author[Bridget Eileen Tenner]{Bridget Eileen Tenner$^\dagger$}
\address{Department of Mathematical Sciences, DePaul University, Chicago, IL, USA}
\email{bridget@math.depaul.edu}
\thanks{$^\dagger$Research partially supported by NSF Grant DMS-2054436.}

\begin{document}

\maketitle

\begin{abstract}
    Previous work has shown that the disarray (or displacement) of an (affine) (signed) permutation is bounded in terms of its Coxeter length.  Here, we characterize the permutations for which the bound is sharp in two ways: in terms of a natural property of their reduced words, and by ``globally'' avoiding the pattern $321$. 
\end{abstract}

\section{Introduction}\label{sec:introduction}

For an (affine) (signed) permutation $w$ with window size $n$, define the \emph{disarray} or \emph{total displacement} $\dis(w)$ to be
\[
\dis(w) := \sum_{i = 1}^n |w(i) - i|.
\]
In \cite{LT}, we showed that for a class of groups that we call \emph{unbranched George groups} (the Coxeter groups of finite types A and B and affine types A and C), the disarray is bounded in terms of the \emph{Coxeter length} $\ell_S$:
\[
\frac{\dis(w)}{2} \leq \lS(w).
\]
This raises the following question.
\begin{question}\label{ques:main question}
    When is $\dis(w)/2$ equal to $\lS(w)$?
\end{question}
In this paper, we answer Question~\ref{ques:main question} in all unbranched George groups, giving a nice connection between the statistics $\dis$ and $\lS$ and a natural notion of pattern avoidance in this context. There are a handful of small cases whose behavior is different from that of general unbranched George groups. This is because the Dynkin diagrams in those small cases are, in a sense, degenerate; they have 
so few generators that the usual diagram shape ``collapses''; see Section~\ref{sec:groups} for details.  Our main result is the following theorem.

\begin{theorem}\label{main theorem}
    For any $w$ in an 
    unbranched George group of window size $n$, 
    the following are equivalent:
    \begin{enumerate}[label={\rm (\Alph*)}]
        \item $\dis(w)/2 = \lS(w)$,\label{conditions:length}
        \item in every reduced word for $w$, for all $i \in [n-1]$, every two copies of $i$ are separated by a copy of both $i-1$ and of $i+1$, and \label{conditions:reduced word new}
        \item $w$ globally avoids the pattern $321$.\label{conditions:pattern}
    \end{enumerate}
    \noindent Moreover, in the nondegenerate cases, the above properties are equivalent to 
    \begin{enumerate}[label={\rm (\Alph*$^\prime$)\negphantom{$^\prime$}},ref={\rm (\Alph*$^\prime$)}]\setcounter{enumi}{1}
        \item
        no reduced word for $w$ contains consecutive letters $i(i\pm 1)i$ for $i \in [n-1]$.\label{conditions:reduced word}
    \end{enumerate}
 \end{theorem}

In Section~\ref{sec:terminology}, we establish notation and terminology and present relevant background on the objects of study. The main theorem is proved in Section~\ref{sec:main theorem proof} by establishing a sequence of equivalences, first in the nondegenerate case and then in the degenerate case. We conclude in Section~\ref{sec:final remarks} by discussing related enumerative questions and by proving a uniform interpretation of the statistic $\dis$ in terms of root heights, and by posing some open questions.

\section{Terminology and previous results}\label{sec:terminology}

\subsection{The groups}\label{sec:groups}

The main objects of study in this paper are the Coxeter groups $S_n$, $S^B_n$, $\affS_n$, and $\affS^C_n$, of finite types A and B and affine types A and C.  As in \cite{LT}, we view these groups as permutation groups on subsets of $\ZZ$ that respect a natural symmetry group on their domain, and call them the \emph{unbranched George groups}.  (The terminology \emph{George group} was coined in \cite{EE}; the adjective \emph{unbranched} refers to their Coxeter--Dynkin diagrams, see Figure~\ref{fig:dynkin}.)  We will carry over the notation from that paper and the curious reader is encouraged to reference \cite[\S2]{LT} for a fuller background on these groups. 

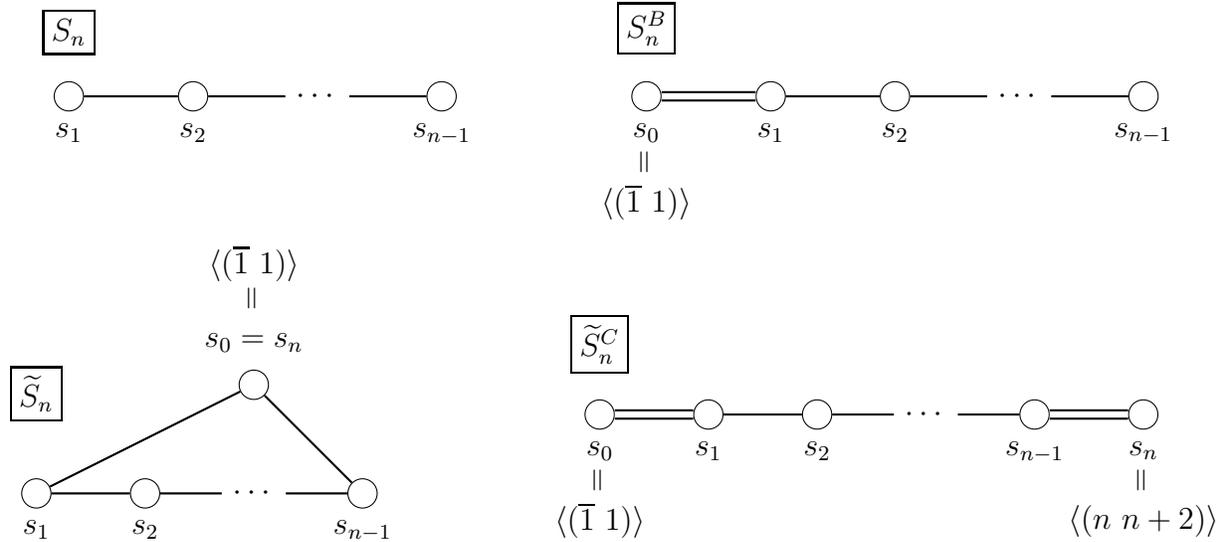
\begin{figure}
\begin{center}
\raisebox{0em}{\begin{tikzpicture}[v/.style={circle,draw}, node distance = 4em]
\node[v] (1)  {};
\node[v] (2) [right of = 1] {};
\node (3) [right of = 2] {$\cdots$} ;
\node[v] (5) [right of = 3] {};
\draw[thick] (1) 
node[below=0.1em] {$\begin{array}{c} s_1 \\ \rotatebox{90}{\phantom{=}} \\ \phantom{\langle (\overline{1} \ 1)\rangle} \end{array}$} 
-- (2) node[below=0.5em] {$s_2$} -- (3) -- 
(5) node[below=0.5em] {$s_{n-1}$}
    ;
    \draw (1) node[above=1em] {$\framebox{$S_n$}$};
\end{tikzpicture}}
\hfill
\begin{tikzpicture}[v/.style={circle,draw}, node distance = 4em]
\node[v] (0)  {};
\node[v] (1) [right of = 0] {};
\node[v] (2) [right of = 1] {} ;
\node (3) [right of = 2] {$\cdots$};
\node[v] (5) [right of = 3] {};
\draw[thick, double, double distance = .2em] (0) 
node[below=0.1em] {$\begin{array}{c} s_0 \\ \rotatebox{90}{=} \\ \langle (\overline{1} \ 1)\rangle \end{array}$} 
[double] -- (1) node[below=0.5em] {$s_1$}; 
\draw[thick] (1) -- (2) node[below=0.5em] {$s_2$} -- (3) -- 
(5) 
node[below=0.1em] {$\begin{array}{c} s_{n-1} \\ \rotatebox{90}{\phantom{=}} \\ \phantom{\langle (n \ n+2)\rangle} \end{array}$};    
    \draw (0) node[above=1em] {$\framebox{$S^B_n$}$};
\end{tikzpicture}

\begin{tikzpicture}[v/.style={circle,draw}, node distance = 3.5em]
\node[v] (1)  {};
\node[v] (2) [right of = 1] {};
\node (3) [right of = 2] {$\cdots$} ;
\node[v] (5) 
            [right of = 3] {};
\node[v] (6) [above of = 3] {};
\draw[thick] (1) node
    [below=0.5em] {$s_1$} -- (2) node[below=0.5em] {$s_2$} -- (3) -- 
(5) node
    [below=0.5em] {$s_{n - 1}$} -- (6) node[above=0.5em] 
    {$\begin{array}{c} \langle (\overline{1} \ 1)\rangle \\
    \rotatebox{90}{=} \\
    s_0 = s_n \end{array}$}
    -- (1);
    \draw (1) node[above=2em] {\framebox{$\affS_n$}};
\end{tikzpicture}
\hfill
\begin{tikzpicture}[v/.style={circle,draw}, node distance = 3.5em]
\node[v] (0)  {};
\node[v] (1) [right of = 0] {};
\node[v] (2) [right of = 1] {} ;
\node (3) [right of = 2] {$\cdots$};
\node[v] (5) [right of = 3] {};
\node[v] (6) [right of = 5] {};
\draw[thick, double, double distance = .2em] (0) 
node[below=0.1em] {$\begin{array}{c} s_0 \\ \rotatebox{90}{=} \\ \langle (\overline{1} \ 1)\rangle \end{array}$} 
[double]
-- (1) node[below=0.5em] {$s_1$}; 
\draw[thick] (1) -- (2) node[below=0.5em] {$s_2$} -- (3) -- 
(5) node[below=0.5em] {$s_{n - 1}$};
\draw[thick, double, double distance = .2em] (5) -- (6) 
node[below=0.1em] {$\begin{array}{c} s_n \\ \rotatebox{90}{=} \\ \langle (n \ n+2)\rangle \end{array}$};
    \draw (0) node[above=1em] {\framebox{$\affS^C_n$}};
\end{tikzpicture}
\end{center}
\caption{The Coxeter--Dynkin diagrams of the nondegenerate unbranched George groups. 
}
\label{fig:dynkin}
\end{figure}

Let $W$ be a George group.  Say that an integer $i$ is \emph{frozen} if $w(i) = i$ for all $w \in W$, and \emph{unfrozen} otherwise.  In particular, the unfrozen integers for the unbranched George groups are listed in Table~\ref{table:frozen elements}.  In the groups $S^B_n$ and $\affS^C_n$, the frozen values act as \emph{mirrors} around which the group elements exhibit the symmetries $w(-i) = -w(i)$ (in both $S^B_n$ and $\affS^C_n$) and $w(2(n + 1) - i) = 2(n + 1) - w(i)$ (in $\affS^C_n$), so that $\affS^C_n$ respects an action of the infinite dihedral group on $\ZZ$.  The groups $S_n$ and $\affS_n$ have no frozen values; the affine symmetric group $\affS_n$ respects the translation symmetry $w(i + n) = w(i) + n$ for all $i \in \ZZ$.

\begin{table}[htbp]
\begin{tabular}{c|c|c}
    Group & Unfrozen elements & Symmetries respected \\ \hline
    $S_n$ & $[n] := \{1, \ldots, n\}$ \\[0.5em]
    $S^B_n$ & $[\pm n] := \{-n, \ldots, -1, 1, \ldots, n\}$ & $w(-i) = -w(i)$ \\[0.5em]
    $\affS_n$ & $\ZZ$ & infinite cyclic: $w(i + n) = w(i) + n$\\[0.5em]
    $\affS^C_n$ & $\ZZ \smallsetminus (n + 1)\ZZ$ & infinite dihedral: $w(-i) = -w(i)$ and \\ &&$w(2(n + 1) - i) = 2(n + 1)-w(i)$
\end{tabular}
\caption{The unfrozen elements in unbranched George groups.}\label{table:frozen elements}
\end{table}

The symmetries respected by each of the groups $S_n$, $\affS_n$, $S^B_n$, $\affS^C_n$ divide the unfrozen values in the domain into $n$ equivalence classes that we call \emph{symmetry classes}.  If $W$ is a George group and $i$ and $j$ are integers such that there exists an element $w \in W$ such that $w(i) = j$ and $w(j) = i$, then we denote by $\langle (i \ j)\rangle$ the extension of the transposition $(i \ j)$ under the symmetry conditions of the group.  Every George group is generated naturally by a set of \emph{simple transpositions}.  In the unbranched George groups, these are as follows: all of the groups $S_n$, $\affS_n$, $S^B_n$, and $\affS^C_n$ contain the adjacent transpositions $s_1, \ldots, s_{n - 1}$ where $s_i = \langle (i \ i + 1)\rangle$.  The affine symmetric group $\affS_n$ contains one additional simple transposition, the adjacent transposition $s_0 = s_n = \langle (0 \ 1)\rangle = \langle (n \ n + 1)\rangle$.  These elements swap two adjacent integers that belong to different symmetry classes.  The groups $S^B_n$ and $\affS^C_n$ contain a second flavor of simple transposition, namely, the ``sign-change transpositions'' $s_0 = \langle (-1 \ 1)\rangle$ (in $S^B_n$ and $\affS^C_n$) and $s_n = \langle (n \ n + 2)\rangle$ (only in $\affS^C_n$) that swap two integers in the same symmetry class across a mirror.  The set of simple transpositions for $W$ form a Coxeter generating set, with the relations encoded by their Coxeter--Dynkin diagram in the usual way.

\begin{definition}\label{defn:reduced words}
Fix a George group $W$ and an element $w \in W$. The \emph{reduced words} of $w$ are the words $i_1\cdots i_{\ell}$ such that $w$ is equal to the product $s_{i_1} \cdots s_{i_{\ell}}$ of generators of $W$ and $\ell = \lS(w)$.
\end{definition}

As suggested in Section~\ref{sec:introduction}, some of the unbranched George groups are small enough that they lose their type-specific properties. We consider those groups to be \emph{degenerate}, and make the following definition.

\begin{definition}\label{defn:unicorn}
    The \emph{nondegenerate unbranched George groups} are $S_n$ for all $n \ge 1$, $S_n^B$ for all $n \ge 1$, $\affS_n$ for all $n \ge 3$, and $\affS^C_n$ for all $n \ge 2$. 
\end{definition}

The omitted nontrivial cases are $\affS_2$ and $\affS^C_1$, both of which are isomorphic to the infinite dihedral group (see Figure~\ref{fig:affS2 Dynkin}).

\begin{remark}\label{rem:weird ones}
 It will be shown in Section~\ref{sec:weird ones} that Theorem~\ref{main theorem} holds in the degenerate cases for a simple reason: in both $\affS_2$ and $\affS^C_1$, \emph{every} element $w$ satisfies conditions~\ref{conditions:length}, \ref{conditions:reduced word new}, and~\ref{conditions:pattern} of the theorem. 
 \end{remark}

\subsection{Pattern avoidance}

The statement of our main theorem involves a natural type of pattern avoidance that, to our knowledge, has not previously been studied in this level of generality. 
Say that two sequences $a_1 \cdots a_n$ and $b_1 \cdots b_n$ of integers are \emph{order-isomorphic} if for all $i, j$, we have $a_i < a_j$ if and only if $b_i < b_j$.

\begin{definition}\label{defn:global pattern}
    Fix a permutation $p \in S_k$. An (affine) (signed) permutation $w$ \emph{globally contains} the pattern $p \in S_k$ if there exist unfrozen integers $i_1 < \cdots < i_k$ such that $w(i_1)\cdots w(i_k)$ is order-isomorphic to $p$. Otherwise $w$ \emph{globally avoids} the pattern $p$.
\end{definition}

For example, adopting the convention of representing elements of $S^B_n$ in one-line notation as $w(1)\cdots w(n)$ and writing $\ol{i}$ to represent $-i < 0$, we have that the signed permutation $w = \ol{1} \ol{2} \in S^B_2$ globally contains the pattern $321$ because $w(-1)  = 1 > w(1) = -1 > w(2) = -2$.

In the case of the symmetric group $S_n$, global pattern containment/avoidance is exactly the usual notion of pattern containment/avoidance for permutations.  In the affine symmetric group $\affS_n$, global pattern avoidance was introduced by Green \cite{Green}.  On the other hand, the ``classical'' notion of pattern avoidance/containment in $S^B_n$ is different from the global one in important ways: it allows the pattern itself to be signed, signs are taken into account when testing containment, and the pattern must be found among the positive positions of the permutation -- not among any positions, as is the case in the global setting.

\begin{definition}\label{defn:classical pattern}
Fix a signed permutation $p \in S^B_k$. A signed permutation $w$ \emph{(classically) contains} the pattern $p$ if there exist integers $1 \le i_1 < \cdots < i_k \le n$ with
\begin{itemize}
    \item $|w(i_1)| \, \cdots \, |w(i_k)|$ order-isomorphic to $|p(1)| \, \cdots \, |p(k)|$, and
    \item $w(i_j) \cdot p(j) > 0$ for all $j$.
\end{itemize}
Otherwise $w$ \emph{(classically) avoids} the pattern $p$.
\end{definition}

When $p \in S_k \subset S^B_k$ and $w \in S_n \subset S^B_n$, Definition~\ref{defn:classical pattern} coincides with the usual notion of pattern containment/avoidance.

\subsection{Previous work}

The main result of this paper is the equivalence of three properties: equality of $\dis(w)/2$ and $\lS(w)$, global avoidance of $321$, and a property of reduced words. That a permutation's global patterns are related to its reduced words is already known in the cases of $S_n$ and $\affS_n$.

\begin{proposition}[\cite{BJS, Green}]
    \label{prop:BJS and Green}
    If $w$ is either a permutation in $S_n$ or an affine permutation in $\affS_n$ then the following are equivalent:
    \begin{itemize}
        \item $w$ globally avoids the pattern $321$,
        \item $w$ is fully commutative, and
        \item no reduced word for $w$ contains consecutive letters $i(i\pm1)i$ for $i \in [n-1]$. 
    \end{itemize}
\end{proposition}

Cases of the third bullet point in Proposition~\ref{prop:BJS and Green} may be moot (e.g., reduced words for permutations in $S_n$ do not use the letters $0$ or $n$, so cannot contain factors $101$ or $(n-1)n(n-1)$); we present the result in this format in order to subsequently draw parallels with other groups.

Moreover, $\dis(w)/2 = \lS(w)$ is already known to be equivalent to these properties in case of $S_n$.

\begin{proposition}[{\cite[Theorem~1.4]{PT}}]\label{prop:PT}
   For $w \in S_n$, the following are equivalent:
   \begin{itemize}
       \item $\dis(w)/2 = \lS(w)$ and
       \item $w$ (globally) avoids $321$.
   \end{itemize}
\end{proposition}

Combining Propositions~\ref{prop:BJS and Green} and~\ref{prop:PT} gives the equivalence of these three properties in the symmetric group, foreshadowing our main result.

The result for signed permutations builds on work of Stembridge, characterizing so-called \emph{fully commutative top elements}. The larger context of these objects is not relevant to the present work, and the reader is referred to \cite{Stembridge} for further details.

\begin{lemma}[{\cite[Theorem~2.1]{Stembridge96} and \cite[Theorem~4.1 and Corollary~5.6]{Stembridge}}]\label{lem:fc top characterization}
    For $w \in S^B_n$, the following are equivalent:
    \begin{itemize}
        \item $w$ is a fully commutative top element,
        \item no reduced word for $w$ contains consecutive letters $i(i\pm 1)i$ for $i \in [n-1]$, and 
        \item $w$ classically avoids the patterns in the set 
    $$ \mathcal{P} := \big\{1 \ol{2} \, , \ \ol{1} \ol{2} \, , \ 321 \, , \ 32\ol{1} \, , \ \ol{3}21 \, , \ \ol{3}2\ol{1}\big\}.$$
    \end{itemize}
\end{lemma}

This allows us to give an almost identical reformulation of Proposition~\ref{prop:BJS and Green} in the case of type B, relating patterns and reduced words in that setting.  A similar statement was asserted in \cite[Chapter~11]{FanThesis}, see Remark~\ref{rem:fan remark}.

\begin{proposition}
    \label{prop:Stembridge}
    If $w$ is a signed permutation in $S^B_n$, then the following are equivalent:
    \begin{itemize}
        \item $w$ globally avoids the pattern $321$,
         \item $w$ is a fully commutative top element (i.e., $w$ classically avoids elements of $\mathcal{P}$), and
       \item no reduced word for $w$ contains consecutive letters $i(i\pm1)i$ for $i \in [n-1]$.
    \end{itemize}  
\end{proposition}

\begin{proof}
First, we show that a signed permutation $w$ globally avoids $321$ if and only if it classically avoids the elements of $\mathcal{P}$ defined in Lemma~\ref{lem:fc top characterization}. 
Let $w$ be a signed permutation.

If $w$ classically contains an element of $\mathcal{P}$, then $w$ necessarily globally contains $321$: if $w(i)w(j)$ forms a 
$1\ol{2}$- or $\ol{1}\ol{2}$-pattern 
(so, in particular, $0 < i < j$), then $w(-j)w(i)w(j)$ is a global $321$-pattern; and if $w(h)w(i)w(j)$ forms a $321$-, $32\ol{1}$-, $\ol{3}21$-, or $\ol{3}2\ol{1}$-pattern 
(again, $0 < h < i < j$), then $\max\{w(h),w(-h)\} w(i)w(j)$ is a global $321$-pattern.

Now suppose that $w(h)w(i)w(j)$ is a global $321$-pattern; that is,
$$h < i < j \hspace{.25in} \text{and} \hspace{.25in} w(h) > w(i) > w(j).$$

There are two cases to consider. First suppose that $\#\{|h|, |i|, |j|\} = 2$, and hence $j > 0$. Then either $|h| \in \{i, j\}$, in which case $w(|i|)w(j)$ is either a $1\ol{2}$-pattern or a $\ol{1}\ol{2}$-pattern; 
or $h < i = - j$, in which case $w(j)w(|h|)$ is a $\ol{1}\ol{2}$-pattern.

Now suppose that 
$\#\{|h|, |i|, |j|\} = 3$. We break the argument into cases based on the signs of $h < i < j$. The symmetry of signed permutations means that $w(-j)w(-i)w(-h)$ is also a global $321$-pattern, reducing four cases to two.

\begin{itemize}
    \item \framebox{$h > 0$}\\
        If $w(i) < 0$, then $w(i)w(j)$ is a $\ol{1}\ol{2}$-pattern. 
        Now suppose that $w(i) > 0$. If $w(i) > |w(j)|$, then $w(h)w(i)w(j)$ is a $321$- or $32\ol{1}$-pattern; 
        otherwise, $w(i)w(j)$ is a $1\ol{2}$-pattern.
    \item \framebox{$h < 0 < i$}\\
        If $|w(i)| < |w(j)|$, then $w(i)w(j)$ is a $1\ol{2}$- or $\ol{1}\ol{2}$-pattern. 
        Otherwise, $|w(i)| > |w(j)|$ and so $w(i) > 0$.  If $|h| > i$ then $w(i)w(|h|)$ is a $1\ol{2}$-pattern. 
        Otherwise, $|h| < i$ and $w(|h|)w(i)w(j)$ is a $\ol{3}21$- or $\ol{3}2\ol{1}$-pattern. 
    \item \framebox{$i < 0 < j$}
        Apply the previous argument to $w(-j)w(-i)w(-h)$.
    \item \framebox{$j < 0$}
        Apply the first argument to $w(-j)w(-i)w(-h)$.
\end{itemize}

Finally, satisfying the reduced word condition follows from Lemma~\ref{lem:fc top characterization}.
\end{proof}

\section{Main results}\label{sec:main theorem proof}

Our goal is to prove the equivalence of the conditions listed in Theorem~\ref{main theorem}. We will do this in a sequence of steps, sometimes showing an implication in the contrapositive form. To begin, we prove a result that will simplify some of the subsequent arguments.

\begin{proposition}
\label{prop:crossing}
    Let $W$ be an unbranched George group $W$, $w \in W$ any element, and $i$ any unfrozen integer.  Then the number of positions $j \in (-\infty, i]$ for which $w(j) > i$ is equal to the number of positions $j \in (i, +\infty)$ for which $w(j) \leq i$ (and both are finite).
\end{proposition}

We remark that, since $S^D_n \subset S^B_n$ and $\affS^D_n \subset \affS^B_n \subset \affS^C_n$, in fact the result is true for all George groups.  Proposition~\ref{prop:crossing} is implicit in \cite{EE}; we provide a proof whose details will be useful later.

\begin{proof}[Proof of Proposition~\ref{prop:crossing}]
Given an unbranched George group $W$, an (affine) (signed) permutation $w \in W$, and any unfrozen integer $i$, let $a_i(w)$ be the number of positions $j \in (-\infty, i]$ for which $w(j) > i$, and let $b_i(w)$ be the number of positions $j \in (i, +\infty)$ for which $w(j) \leq i$, so that the assertion to be proved is that $a_i(w) = b_i(w) \in \NN$ for every $w$.  This claim is obviously true for the identity.  Suppose the claim is true for an affine permutation $v$, and let us consider whether it is true for $v \cdot s$ where $s$ is a simple transposition and $\lS(v\cdot s) > \lS(v)$.  

If $s$ is \emph{not} the (unique if it exists) simple transposition that transposes $i$ with the next-largest unfrozen value then $s$ does not transpose any entry whose position is in $(-\infty, i]$ with an entry whose position is in $(i, +\infty)$ and so $\{ v \cdot s (j) \colon j \in (-\infty, i]\} = \{ v (j) \colon j \in (-\infty, i]\}$ and $\{ v \cdot s (j) \colon j \in (i, +\infty)\} = \{ v (j) \colon j \in (i, +\infty)\}$.  Thus in this case $a_i(v\cdot s) = a_i(v) = b_i(v) = b_i(v \cdot s)$, by the inductive hypothesis for $v$.  If, on the other hand, $s$ is the simple transposition that transposes position $i$ with the next-largest unfrozen value $i'$ (either $i' = i + 1$ or, if $i + 1$ is the position of a mirror, then $i' = i + 2$), we have three cases:
  \begin{enumerate}[label=(\alph*)]
        \item if $v(i)$ and $v(i')$ are both larger than $i+1/2$ or both smaller than $i+1/2$, then $a_i(v \cdot s) = a_i(v) = b_i(v) = b_i(v \cdot s)$;
        \item if $v(i) < i+1/2 < v(i+1)$,
        then $a_i(v \cdot s) = 1 + a_i(v) = 1 + b_i(v) = b_i(v \cdot s)$; and
        \item $v(i) > i+1/2 > v(i+1)$ cannot occur because it would imply $\lS(v\cdot s) < \lS(v)$.
  \end{enumerate}
This proves that $a_i(w) = b_i(w)$ for all $w$, by induction.
\end{proof}

We now prove that if an element satisfies property~\ref{conditions:pattern} of Theorem~\ref{main theorem}, then it also satisfies property~\ref{conditions:length} of the theorem.

\begin{theorem}
\label{thm:321-free implies equality}
    Let $W$ be an unbranched George group and $w \in W$.  If $w$ globally avoids the pattern $321$, then $\dis(w)/2 = \lS(w)$.  
\end{theorem}

\begin{proof}
We proceed by induction on $\lS(w)$.  Certainly the statement is true for the identity permutation.  So choose an element $w \in W$ that globally avoids $321$ with $\lS(w) > 0$, and assume that the result is true of all globally $321$-avoiding (affine) (signed) permutations in $W$ of smaller Coxeter length.  Choose a right descent $s$ of $w$, and let $v = w \cdot s$.  Then $v$ also globally avoids $321$ (because $s$ is a descent, so multiplying by it removes a left inversion and creates no new inversions), and consequently by assumption $\lS(w) = \lS(v) + 1 = \dis(v)/2 + 1$.  Then it suffices to show that $\dis(v)/2 + 1 = \dis(w)/2$.  We consider two cases, depending on whether the integers transposed by $s$ belong to the same symmetry class or not.  

Write $s = \langle (i \ i')\rangle$ where $i < i'$ are consecutive non-frozen integers.  If $i$ and $i'$ belong to different symmetry classes then $i' = i + 1$ and
    \begin{equation}\tag{$*$}
        \label{eq:dis diff}
        \dis(w) - \dis(v) = |v(i) - (i + 1)| + |v(i + 1) - i| - |v(i + 1) - (i + 1)| - |v(i) - 1|.
    \end{equation}
Let us consider now the sizes of the values $v(i)$ and $v(i + 1)$.  Since $s = \langle (i \ i + 1)\rangle$ is a descent of $w = v \cdot s$, it is an ascent of $v$, i.e., $v(i) < v(i + 1)$.  In the taxonomy that appears in the proof of Proposition~\ref{prop:crossing}, we claim that case (a) is not possible: if it were the case that $v(i) > i$ then (by definition, since $v(i + 1) > v(i) > i$) $a_i(v \cdot s) = a_i(v) > 0$ and so (by Proposition~\ref{prop:crossing}) $b_i(v \cdot s) = b_i(v) > 0$.  Then there would be some unfrozen integer $k > i$ for which $v(k) \leq i$.  Since $v(k) \leq i$ and $v(i + 1) > i$, $k \neq i + 1$, and therefore $k > i + 1$.  Therefore the sequence $w(i) w(i + 1)w(k) = v(i + 1) v(i) v(k)$ is an instance of $321$ in $w$, a contradiction.  A similar argument establishes that $v(i + 1) > i$ (assuming the negation, we find $k < i$ with $v(k) > i$).  Plugging these inequalities into \eqref{eq:dis diff}, we have that
\begin{align*}
\frac{1}{2}\dis(w) & = \frac{1}{2}\big(\dis(v) + (i + 1 - v(i)) + (v(i + 1) - i) - (v(i + 1) - (i + 1)) - (i - v(i))\big) \\
& = \frac{1}{2} \dis(v) + 1 = \lS(w),
\end{align*}
as claimed.

In the alternative case that $s = \langle (i \ i + 2 ) \rangle$ transposes two positions in the same symmetry class across a mirror at position $i + 1$, in place of \eqref{eq:dis diff} we have
\[
\dis(w) - \dis(v) = |v(i + 2) - i| - |v(i) - i|.
\]
Because $i + 1$ is a mirror, we have $v(i + 2) = 2(i + 1) - v(i)$.  Moreover, because $s$ is an ascent of $v$, we have $v(i) < v(i + 2)$, and  therefore $v(i) \leq i < v(i + 2)$.  Thus in this case
\begin{align*}
    \frac{1}{2} \dis(w) & = \frac{1}{2}\big(\dis(v) + ((2(i + 1) - v(i)) - i) - (i - v(i))\big) \\
    & = \frac{1}{2}\dis(v) + 1 = \lS(w),
\end{align*}
as claimed.
\end{proof}

For the remainder of the proof of Theorem~\ref{main theorem}, we split into two cases, depending on whether the group under consideration is degenerate or not.

\subsection{The nondegenerate case}

For the next three results, we restrict ourselves to the case that $W$ is a nondegenerate unbranched George group.  The first of these is that property~\ref{conditions:reduced word} of Theorem~\ref{main theorem} implies property~\ref{conditions:pattern}.  We state it here in the contrapositive form.

\begin{theorem}\label{thm: reduced word implies 321-free}
    Let $W$ be a nondegenerate unbranched George group of window size $n$. If 
    an element $w \in W$ globally contains the pattern $321$, then some reduced word for $w$ contains a consecutive triple $x (x \pm 1) x$ for $x \in [n - 1]$.
\end{theorem}

\begin{proof}
If $W$ is one of $S_n$, $S^B_n$, or $\affS_n$, the result follows immediately from Proposition~\ref{prop:BJS and Green} or~\ref{prop:Stembridge} (i.e., from the work of Billey--Jockusch--Stanley, Green, and Stembridge).  Thus it suffices to consider the case that $W = \affS^C_n$ is the group of affine signed permutations, $n \geq 2$.

Suppose that $w \in \affS^C_n$ globally contains the pattern $321$, and choose an instance $i < j < k$, $w(i) > w(j) > w(k)$ of $321$ in $w$ so that its \emph{width} $k - i$ is minimal among all global copies of $321$ in $w$.  It will be convenient to take advantage of the following symmetry: if we specify a frozen value $f$ in the interval $[i, k]$, we claim that without loss of generality we may take $f = 0$.  There are two cases: if $f = 2m(n + 1)$ for some $m \in \ZZ$, then by the translation symmetry of $\affS^C_n$ we can replace the triple $(i, j, k)$ with the triple $(i - 2m(n + 1), j - 2m(n + 1), k - 2m(n + 1))$ and replace $f$ with $0$.  If instead $f = (2m + 1)(n + 1)$ for some $m \in \ZZ$, we may apply the group automorphism $\alpha: \affS^C_n \to \affS^C_n$ that is defined for $u \in \affS^C_n$ by $\alpha(u)(x) = n + 1 - u(x)$ for all $x \in \ZZ$.  This automorphism is induced by the unique nontrivial symmetry $s_\ell \leftrightarrow s_{n - \ell}$ for $\ell = 0, 1, \ldots, n$ of the Coxeter--Dynkin diagram of $\affS^C_n$, so that $\alpha(w)$ has a reduced word with one of the forbidden patterns if and only if $w$ does; moreover, it is easy to check that $(n + 1 - k, n + 1 - j, n + 1 - i)$ is an instance of the global pattern $321$ in $\alpha(w)$ and that $-2m(n + 1)$ is a frozen value in the interval spanned by this instance in $\alpha(w)$, so then by the prior case we can choose another instance for which the selected frozen value is $0$.

We proceed by simultaneous induction on the width $k - i$ and the number of symmetry classes of unfrozen elements in $\{i + 1, \ldots, k - 1\}$.  Since this latter set contains at least one unfrozen element (namely, $j$), the base cases are when there is exactly one symmetry class of unfrozen elements present in $\{i + 1, \ldots, k - 1\}$.  There are several ways this can happen: 
     \begin{itemize}
         \item if $i = j - 1 = k - 2$, then $i$, $j$, and $k$ belong to distinct symmetry classes (they cannot be translations because $w(i)w(j)w(k)$ is decreasing, and they cannot be on opposite sides of a mirror because there would have to be a frozen value separating them).  In this case, let $s_x$ and $s_y$ be the simple transpositions that respectively swap $i$ with $j = i + 1$ and $j$ with $k = j + 1$.  Then $x, y \in [n - 1]$ because $\{i, i + 1, i + 2\}$ are all unfrozen, $y = x \pm 1$, and $w$ has a reduced word of the form $\cdots \ x \ y \ x$, as desired.
         \item if $k = i + 3$ then one of $i + 1$, $i + 2$ is frozen. Using the reduction of the previous paragraph to take this frozen value to be $0$ and checking both cases, 
         one has necessarily that $w(-2) > w(-1) > 0 > w(1) > w(2)$ and so that $w$ has a reduced word of the form $\cdots 1 \ 0 \ 1$.
         \item if $k = i + 4$ then $i + 2 = 0$. The minimality of $k-i$ 
         means that $w(-2) > w(1) > 0 > w(-1) > w(2)$, in which case $w$ has a reduced word of the form $\cdots 1 \ 0 \ 1$.
     \end{itemize}
     
     Now suppose $w$ globally contains $321$ but is \emph{not} in any of those base cases, so that at least two symmetry classes of unfrozen values are represented by $\{i + 1, \ldots, k - 1\}$ when $k - i$ is chosen minimal among all possible copies of $321$.  Let $i'$ be the minimum unfrozen value larger than $i$ and let $k'$ be the maximum unfrozen value smaller than $k$.  
     
     Suppose first that $i'$ and $k'$ belong to different symmetry classes.  Thus (at least) one of them is in a different symmetry class from $j$; if it is $i'$, then let $s$ be the simple transposition that swaps $i$ and $i'$, otherwise let $s$ be the simple transposition that swaps $k'$ and $k$.  In either case, $w \cdot s$ has a narrower global $321$-pattern than $w$ had had, and that pattern spans the same or fewer symmetry classes of unfrozen values than $w(i)w(j)w(k)$ had done in $w$.  
     So, by induction, $w\cdot s$ has a reduced word that contains a consecutive triple $x \ (x \pm 1) \ x$ for some $x \in [n-1]$.  Furthermore, $w(i') < w(i)$ (or else $(i', j, k)$ would be a narrower copy of $321$ than $(i, j, k)$ in $w$), so $w \cdot s$ is less than $w$ in weak order.  Thus every reduced word for $w\cdot s$ is a consecutive substring of a reduced word for $w$, and so $w$ has a reduced word that contains a consecutive triple $x (x \pm 1) x$ for some $x \in [n-1]$.
     
     Now suppose, instead, that $i'$ and $k'$ belong to the same symmetry class.  We consider two cases, depending on whether they are mirror images of each other or translations.  
     
     If $i'$ and $k'$ are mirror images, then by the discussion at the beginning of the proof, we have $i' = -k'$ and $i = -k$.  If $j$ is not equal to $i'$ or $k'$, then let $s$ be the transposition swapping $i$ with $i'$ (and also $k$ with $k'$). Then $w \cdot s$ has a narrower global $321$-pattern than $w$ had had, and we can proceed by induction.  If $j \in \{i',k'\}$, then assume, without loss of generality, that $j = i'$.  The minimality of $k-i$ means that $w(i') < 0 < w(k')$ (otherwise we would have considered the narrower global $321$-pattern $w(i)w(i')w(k')$).  Choose $j'$ any unfrozen element of $\{i' + 1, \ldots, k' - 1\}$. Such a $j'$ exists because we are not in the base case of the induction.  If $w(j') < w(i')$ then $w(i)w(i')w(j')$ would be a narrower global $321$-pattern in $w$, and if $w(j') > w(k')$ then $w(j')w(k')w(k)$ would a narrower global $321$-pattern in $w$. Thus, by the minimality of $k-i$, we must have $w(i') < w(j') < w(k')$.  But now we can replace $j$ with $j'$, and consider the global $321$-pattern $w(i)w(j')w(k)$, with $j' \not\in \{i', k'\}$, and follow the previous construction.

    Finally, consider the possibility that $i'$ and $k'$ are translations of each other.  If $j \not\in \{i', k'\}$, then proceed as in the case where $i'$, $k'$ are in different symmetry classes. On the other hand, if, without loss of generality, $j = i'$ then $i' < k - (2n + 2)$ and $w(k - (2n + 2)) < w(k)$ so $w(i)w(i')w(k - (2n + 2))$ is a narrower global $321$-pattern in $w$. 
    This would contradict the minimality of $k-i$, so it cannot exist.
\end{proof}

For the next step in our proof of Theorem~\ref{main theorem}, we prove
that property~\ref{conditions:length} implies property~\ref{conditions:reduced word}; we again state the result in contrapositive form.

\begin{theorem}\label{thm: equality implies reduced word}
    Let $W$ be a nondegenerate unbranched George group of window size $n$.  If an element $w \in W$ has a reduced word that contains a consecutive triple $x (x \pm 1) x$ for $x \in [n - 1]$, then $\dis(w)/2 < \lS(w)$.
\end{theorem}
\begin{proof}
If $s_i$ and $s_{i + 1}$ are both adjacent transpositions, we have
\[
s_i s_{i + 1} s_i = s_{i + 1}s_i s_{i + 1} = \langle (i \ i + 2)\rangle
\]
where $i$ and $i + 2$ lie in different symmetry classes, so $\lS(s_i s_{i + 1} s_i) = 3 > 2 =  \dis(s_i s_{i + 1} s_i)/2$.

When $s_1$ is an adjacent transposition and $s_0$ is a sign-change, we have
\[
s_1 s_0 s_1 = \langle (- 2 \ 2) \rangle
\]
with $-2$ and $2$ in the same symmetry class, so $\lS(s_1s_0s_1) = 3 > 2 =  \dis(s_1 s_0s_1)/2$.  Similarly, in $\affS^C_n$, we have that $s_{n - 1} = \langle (n - 1 \ n) \rangle$ is an adjacent transposition, $s_n = \langle (n \ n + 2)\rangle$ is a transposition across the frozen value $n + 1$, and
\[
s_{n - 1} s_n s_{n - 1} = \langle (n - 1 \ n + 3) \rangle
\]
with $n-1$ and $n + 3$ in the same symmetry class. Thus
$$\lS(s_{n - 1}s_ns_{n - 1}) = 3 > 2 =  \dis(s_{n - 1} s_ns_{n - 1})/2.$$

By hypothesis, there is a reduced word for $w$ of the form $\cdots s_i s_{i \pm 1} s_i \cdots$.  The preceding paragraph shows that in all cases, $\dis(s_i s_{i \pm 1} s_i)/2 = 2$.  Considering this reduced word as a factorization of $w$ into a product of $\lS(w) - 3$ simple transpositions and the additional factor  $s_i s_{i \pm 1} s_i$, we have by subadditivity of disarray that
\[
\frac{1}{2} \dis(w) \leq (\lS(w) - 3) \cdot 1 + 2 < \lS(w), 
\]
as claimed.
\end{proof}

Combining Theorems~\ref{thm:321-free implies equality},~\ref{thm: reduced word implies 321-free}, and~\ref{thm: equality implies reduced word} implies that properties~\ref{conditions:length}, \ref{conditions:reduced word}, and~\ref{conditions:pattern} of Theorem~\ref{main theorem} are equivalent in the case of nondegenerate unbranched George groups. What remains to be proved of Theorem~\ref{main theorem} in this case is to show that properties~\ref{conditions:reduced word} and~\ref{conditions:reduced word new} are equivalent, as well.

\begin{theorem}\label{thm:alt}
    If $w$ is an element of a nondegenerate unbranched George group of window size $n$, then no reduced word for $w$ contains consecutive letters $i(i \pm 1)i$ for $i \in [n - 1]$ if and only if, in every reduced word for $w$, for all $i \in [n - 1]$, every two copies of $i$ are separated by both a copy of $i - 1$ and of $i + 1$.
\end{theorem}

\begin{proof}
The backwards direction of the result, that property~\ref{conditions:reduced word new} implies property~\ref{conditions:reduced word}, is immediate: the consecutive triple $i (i \pm 1)i$ contains two copies of $i$ separated by only one of $i - 1$, $i + 1$, and in any nondegenerate group the transpositions $s_{i - 1}$ and $s_{i + 1}$ are not equal for any $i$.

Conversely, suppose that $w$ fails to have property~\ref{conditions:reduced word new}. Suppose that $w$ has a reduced word with two copies of $i$ separated by $i+1$ and not by $i-1$, for some $i \in [n-1]$. Moreover, assume that there are no other copies of $i$ between the two identified ones (otherwise, choose a closer pair of $i$s). We can perform a sequence of commutation moves to relocate any letters less than $i-1$ so that they do not lie between the identified copies of $i$. Thus, the letters lying strictly between these two copies of $i$ are all greater than $i$. If there are two copies of $i+1$ among these letters, then those two copies must be separated by $i+2$, and we know that they have no $i$ between them. Thus those two copies of $i+1$ would also have demonstrated that $w$ fails to have property~\ref{conditions:reduced word new}, and we could instead have started with them. By repeating this argument, we can assume that there is exactly one copy of $i+1$ lying between our two copies of $i$. Now perform a sequence of commutation moves to position these three letters --- $i$, $i+1$, and $i$ --- consecutively in a reduced word for $w$, meaning that $w$ does not have property~\ref{conditions:reduced word}. 

The case in which two copies of $i$ do not have a copy of $i + 1$ between them can be handled identically to the previous argument, swapping the roles of $i - 1$ and $i + 1$, etc.
\end{proof}

\subsection{The 
infinite dihedral cases
}\label{sec:weird ones}

It remains to show the assertion of Remark~\ref{rem:weird ones}, namely, that in the degenerate cases $\affS_2$ and $\affS^C_1$, \emph{every} element satisfies conditions~\ref{conditions:length}, \ref{conditions:reduced word new}, and~\ref{conditions:pattern}.  
The groups $\affS_2$ and $\affS^C_1$ are infinite dihedral groups, and their Coxeter--Dynkin diagrams are shown in Figure~\ref{fig:affS2 Dynkin}.  We begin with condition~\ref{conditions:reduced word new}.

In $\affS^C_1$, $n = 1$, and so condition~\ref{conditions:reduced word new} (which asserts something about all $i \in [n - 1]$) is vacuously true for every element in this group.  Now consider any element $w$ in $\affS_2$.  Since $n = 2$, the assertion to be proved is that in every reduced word for $w$, every two copies of $1$ are separated by a copy of $0$ and by a copy of $2$.  Since $s_0 = s_2$ is the only simple transposition other than $s_1$ in $\affS_2$, this assertion becomes that every two copies of $1$ are separated from each other in every reduced word, and this follows immediately from the definition of reduced words.

We now consider conditions~\ref{conditions:length} and~\ref{conditions:pattern}.  The group $\affS_2$ permutes $\ZZ$ with no frozen values, while $\affS^C_1$ permutes the odd integers with every even integer frozen.  If $f: \ZZ \to (1 + 2\ZZ)$ is defined by $f(z) = 2z + 1$ then $w \mapsto f^{-1} \circ w \circ f$ is an isomorphism from $\affS^C_1$ to $\affS_2$, and one can check that furthermore it preserves $\lS$, $\dis$, and global $321$-patterns.  Thus it suffices to consider $\affS_2$. 

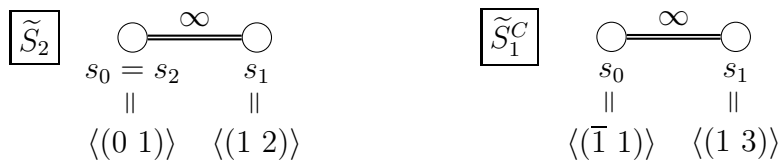
\begin{figure}[htbp]
\begin{tikzpicture}[v/.style={circle,draw}, node distance = 4em]
\node[v] (0)  {};
\node[v] (1) [right of = 0] {};
\draw[thick] (0) node[below=0.1em] {$\begin{array}{c} s_0 = s_2 \\ \rotatebox{90}{=} \\ \langle (0 \ 1)\rangle \end{array}$} 
[double] -- node[above]{$\infty$} (1) node[below=0.1em] {$\begin{array}{c} s_1 \\ \rotatebox{90}{=} \\ \langle (1 \ 2)\rangle \end{array}$}; 
\draw (0) node[left=2em] {\framebox{$\affS_2$}};
\end{tikzpicture}
\hspace{4em}
\begin{tikzpicture}[v/.style={circle,draw}, node distance = 4em]
\node[v] (0)  {};
\node[v] (1) [right of = 0] {};
\draw[thick] (0) node[below=0.1em] {$\begin{array}{c} s_0 \\ \rotatebox{90}{=} \\ \langle (\overline{1} \ 1)\rangle \end{array}$} 
[double] -- node[above]{$\infty$} (1) node[below=0.1em] {$\begin{array}{c} s_1 \\ \rotatebox{90}{=} \\ \langle (1 \ 3)\rangle \end{array}$}; 
\draw (0) node[left=2em] {\framebox{$\affS^C_1$}};
\end{tikzpicture}

    \caption{The Coxeter--Dynkin diagrams of $\affS_2$ and $\affS^C_1$}
    \label{fig:affS2 Dynkin}
\end{figure}

For any $w \in \affS_2$, we have $\cdots < w(-3) < w(-1) < w(1) < w(3) < \cdots$ and $\cdots < w(-2) < w(0) < w(2) < \cdots$, so the infinite word $\cdots w(-2)w(-1)w(0)w(1)w(2)\cdots$ is a union of two increasing subsequences. Consequently, it has no decreasing subsequence of length longer than $2$, meaning that $w$ globally avoids $321$.  
The proof of Theorem~\ref{thm:321-free implies equality} is valid in $\affS_2$, so every element in $\affS_2$ satisfies $\dis(w)/2 = \lS(w)$.

\section{Further remarks}\label{sec:final remarks}

\subsection{Enumeration}

The enumeration of globally $321$-avoiding elements has been carried out in all of the unbranched George groups, under various guises.  We briefly summarize the state of this work below.

The enumeration of $321$-avoiding permutations was carried out by MacMahon \cite[\S97]{MacMahon}: the number of $321$-avoiding permutations in $S_n$ is the Catalan number $C_n$.  This enumeration can be further refined by length, and Cheng--Elizalde--Kasraoui--Sagan gave a detailed analysis of the generating polynomial
\[
C_n(q) := \sum_{w \in S_n: w \text{ avoids } 321} \hspace{-.25in} q^{\lS(w)}.
\]
They showed that it satisfies a $q$-analogue 
\begin{equation}\label{eq:C recurrence}
C_{n + 1}(q) = C_n(q) + \sum_{k = 0}^{n - 1} q^{k + 1} C_k(q) C_{n - k}(q),
\end{equation}
of the usual Catalan recurrence \cite[Theorem 1.1]{CEKS} by establishing statistic-preserving bijections with certain colored Motzkin paths 
and certain polyominoes.

There are numerous formul\ae\ in the literature for the generating function
\[
S(x) = \sum_{n \geq 0} C_n(q) x^n
\]
that counts $321$-avoiding permutations in all symmetric groups simultaneously.  These include a simple functional equation \cite[Equation~(19)]{CEKS}, various continued fraction formul\ae\ \cite[\S7]{CEKS}, and various formul\ae\ that express $S(x)$ as a ratio of $q$-hypergeometric series \cite[p.~6]{BDPP}, \cite[Theorem~1.1]{BB-MJN} or as a ratio of two series involving $q$-binomial coefficients \cite[Theorem~6.1]{BB-MJN}.

These results have been extended to the other unbranched George groups to varying degrees.  By Proposition~\ref{prop:Stembridge}, signed permutations avoiding global $321$-patterns are exactly Stembridge's fully commutative top elements.  (In \cite{BJN2015}, these are the called the \emph{fully commutative elements corresponding to alternating heaps}.) 
Those elements are enumerated by the central binomial coefficient $\binom{2n}{n}$, sometimes called the ``type B Catalan number'' \cite[Proposition~5.9(b)]{Stembridge}. Let
\[
C^B_n(q) := \sum_{w \in S^B_n: w \text{ globally avoids } 321} \hspace{-.5in} q^{\lS(w)}
\]
be the generating polynomial for global $321$-avoiding elements in $S^B_n$ by length and let
\[
S^B(x) := \sum_{n \geq 0} C^B_{n}(q) x^n
\]
be the generating function for these generating functions.  In \cite{BJN2015}, several implicit formul\ae\ are given for $S^B(x)$, from which one can extract the recurrence relation
\begin{equation}\label{eq:B catalan recurrence}
C^B_n(q) = (1 - q^{n + 1})C_n(q) + \sum_{k = 0}^n (q^{n - k + 1} + q^{k + 1}) C_{n - k}(q) \cdot C^B_{k}(q),
\end{equation}
a $q$-analogue of the classic $q = 1$ recurrence 
$$\binom{2(n + 1)}{n + 1} = 2 \sum_{k = 0}^n C_{n - k} \cdot \binom{2k}{k}$$
for central binomial coefficients.
In \cite[\S1.3]{BJN2015}, a statistic-preserving bijection is given between globally $321$-avoiding elements of $S^B_n$ (encoded in terms of a heap structure on their reduced words) and certain colored Motzkin-type paths. 
Explicit formul\ae\ for the generating function $S^B(x)$ are given in \cite[(5) and (36)]{BB-MJN}, the latter in the flavor of a ratio of $q$-hypergeometric series.  

The enumeration for fully commutative elements in $\affS_n$ according to length was initially carried out in \cite{HJ}.  This has since been reworked in various ways \cite{AH, BJN2015, BB-MJN, BJN2019}, including a bijection with colored Motzkin-type paths on a cylinder \cite[\S2.3]{BJN2015} and a formula \cite[Theorem~5.1]{BB-MJN} for the generating function 
\[
\wt{S}(x) := \sum_{n \geq 1} \left( \sum_{w \in \affS_{n}: w \text{ avoids } 321} \hspace{-.25in} q^{\lS(w)}\right)x^n 
\]
in terms of $q$-hypergeometric series.

In \cite{BJN2015}, there is a detailed analysis of fully commutative elements in $\affS^C_n$, including a five-part classification of the associated heaps.  It follows from Theorem~\ref{thm:alt} that the globally $321$-avoiding elements belong to the class that correspond to the \emph{alternating heaps} in \cite{BJN2015}.  These may again be encoded in terms of certain colored Motzkin-type walks 
\cite[p.~26]{BJN2015}.  In \cite[Theorem~6.2]{BB-MJN}, a (not simple, but explicit) formula is given for the generating function
\[
\wt{S}^C(x) := \sum_{n \geq 1} \left( \sum_{w \in \affS^C_{n}: w \text{ avoids } 321} \hspace{-.25in} q^{\lS(w)}\right)x^n
\]
(under the name $L(x)$), as a combination of several different flavors of $q$-series.

On one hand, the preceding results may be considered a complete account of the enumeration of globally $321$-avoiding elements of unbranched George groups.  On the other hand, there are a few natural questions that remain.

\begin{question}
Is there a natural continued fraction for $S^B(x)$, analogous to those found in \cite[\S7]{CEKS} for the symmetric group?  What about affine types?    
\end{question}

\begin{question}
  Are there ``Catalan-like'' recurrences for the generating functions counting globally $321$-avoiding elements of $\affS_n$ and $\affS^C_n$, analogous to \eqref{eq:C recurrence} and \eqref{eq:B catalan recurrence}?    
\end{question}

\begin{question}
    Is there a uniform presentation of explicit formul\ae\ for the generating functions $S(x)$, $S^B(x)$, $\wt{S}(x)$, $\wt{S}^C(x)$?  In particular, is there a ``nice'' explicit formula for $\wt{S}^C(x)$?
\end{question}

\subsection{Root height}\label{sec:root height}

In \cite[Rem.~2.5]{PT}, it was observed that the depth of a positive root for $S_n$ (the smallest number $k$ such that the root can be written as $s_{i_k} \cdots s_{i_1}(\alpha)$ for a simple root $\alpha$ and simple transpositions $s_{i_1}$, \ldots, $s_{i_k}$) is equal to its \emph{height} (the sum $c_1 + \cdots + c_n$ of the coefficients when the root is written as a linear combination $\sum c_i \alpha_i$ of simple roots). 
It was pointed out to us by Nathan Reading that for a reflection in the hyperoctahedral group $S^B_n$, our cost function $\dis(t)/2$ is precisely the same as the height of the associated root in the type B root system.  
We extend this observation now, beginning with some background from \cite[Chapter 4]{BB} and \cite[Chapters 5--6]{H}.

Let $(W, S)$ be any Coxeter group, with $S$ the set of simple reflections generating $W$.  For each pair $s, s'\in S$, let $m_{s, s'}$ be the order of $ss'$; i.e., the label on the edge joining $s$ to $s'$ in the Coxeter--Dynkin diagram.  For each $s \in S$, define a formal symbol $\alpha_s$ (a \emph{simple root}), and let $V$ be the $\RR$-span of $\{ \alpha_s \colon s \in S\}$.  Define a symmetric bilinear form $( -\mid - )$ on $V$ by setting
\[
(\alpha_s\mid \alpha_{s'}) = -\cos \frac{\pi}{m_{s, s'}}
\]
and extending by bilinearity.  In particular, we have $(\alpha_s\mid \alpha_s) = 1$ and that $(\alpha_s\mid \alpha_{s'}) = 0$ if $s$ and $s'$ commute.  It turns out that one may define a faithful action of $W$ on $V$ by setting
\begin{equation}\label{eq:refn}
s(\beta) = \beta - 2 (\alpha_s\mid \beta) \alpha_s
\end{equation}
and extending by multiplication \cite[Theorems~4.2.2 and~4.2.7]{BB}.  This representation is the \emph{standard geometric representation} of $W$.

The \emph{standard root system} of $W$ is the set $\{ w(\alpha_s) \colon w \in W, s \in S\}$ (a subset of $V$); its elements are called \emph{roots}.  As in the finite case, every root is either a linear combination of simple roots with all coefficients nonnegative (in which case we call it a \emph{positive root}), or a linear combination of simple roots with all coefficients nonpositive.  The positive roots are in bijection with the reflections in $W$, with the relationship between a reflection $t$ and its corresponding positive root $\alpha_t$ being given by the same formula \eqref{eq:refn} (replacing $s$ with $t$).  Typically, the roots are not \emph{integer} linear combinations of the simple roots, because the coefficients $(\alpha_s\mid \alpha_s')$ that arise when we apply group elements may be irrational.  For example, in the case of $S^B_2$, this construction gives $V = \RR^2$, $(-\mid -)$ the usual dot product, and simple roots $\alpha_0 = (1, 0)$ and $\alpha_1 = (-1/\sqrt{2}, 1/\sqrt{2})$; the other positive roots are
\[
(0, 1) = \alpha_0 + \sqrt{2} \alpha_1 \qquad\text{ and }\qquad (1/\sqrt{2}, 1/\sqrt{2}) = \sqrt{2}\alpha_0 + \alpha_1.
\]
However, there is a certain condition 
that allows us to rescale the lengths of the simple roots to produce an alternative basis for $V$ that spans a lattice stabilized by $W$.
\begin{definition}
    A Coxeter group $(W, S)$ is \emph{crystallographic} if 
    \begin{itemize}
    \item $m_{s, s'} \in \{2, 3, 4, 6, \infty\}$ for all $s \neq s' \in S$, 
    \item the number of edges with label $4$ in any cycle of the Coxeter--Dynkin diagram is even, and
    \item the number of edges with label $6$ in any cycle of the Coxeter--Dynkin diagram is even.
    \end{itemize}
\end{definition}
\noindent In particular, the George groups are all crystallographic.

As explained in \cite[\S6.6]{H}, when $W$ is crystallographic, there exists a scalar $c_s \in \RR_{> 0}$ for each $s \in S$ such that
\begin{itemize}
    \item $c_s = c_{s'}$ if $m_{s, s'} = 3$ or $\infty$,
    \item $c_s/c_{s'} \in \{\sqrt{2}, 1/\sqrt{2}\}$ if $m_{s, s'} = 4$, and
    \item $c_s/c_{s'} \in \{\sqrt{3}, 1/\sqrt{3}\}$ if $m_{s, s'} = 6$.
\end{itemize}
Fix one such collection of $c_s$.  Then defining $\lambda_s = c_s \alpha_s$ for all $s \in S$, we have (after a calculation) that for all $s, s' \in S$,
\[
s(\lambda_{s'}) = \lambda_{s'} + d_{s, s'} \lambda_s
\]
for an \emph{integer} $d_{s, s'}$; concretely,
\begin{equation}\label{eq:d}
d_{s, s'} = \begin{cases}
        0 & \text{ if } m_{s, s'} = 2 \text{ (no edge)}, \\
        1 & \text{ if } m_{s, s'} = 3 \text{ (in which case }  c_s = c_{s'}), \\
        1 & \text{ if } m_{s, s'} = 4 \text{ or } 6 \text{ and } c_s > c_{s'}, \\
        2 & \text{ if } m_{s, s'} = \infty \text{ or } m_{s, s'} = 4 \text{ and } c_{s'} = \sqrt{2} c_{s}, \text{ or} \\
        3 & \text{ if } m_{s, s'} = 6 \text{ and } c_{s'} = \sqrt{3} c_{s}.
\end{cases}
\end{equation}
It follows that $\{\lambda_1, \ldots, \lambda_n\}$ is a basis for $V$ whose $\ZZ$-span is stabilized by the action of $W$.  Taking these new vectors $\{\lambda_i\}$ as the simple roots, and $\Phi := \{ w(\lambda_s) \colon w \in W, s \in S\}$ the associated root system, it remains true that there is a bijection between reflections $t$ and positive roots $\lambda_t$, with $\lambda_t$ perpendicular to the fixed space of $t$; but now furthermore the height $\height(\lambda)$ of each root $\lambda$ (the sum of its coefficients when written as a linear combination of simple roots) is an integer.  Then we can define the height $\height(w)$ of each element $w$ of $W$ as the cheapest cost of a reflection factorization of $w$, where the cost of a reflection $t$ is $\height(\lambda_t)$.  Since simple reflections all cost $1$, we have for sure that $\height(w) \leq \lS(w)$ for all $w \in W$.  Our next result is an analogue of Theorem~\ref{thm: equality implies reduced word}, valid in any crystallographic Coxeter group.
\begin{proposition}\label{prop:general Coxeter}
    Fix a crystallographic Coxeter group $(W, S)$ with root system $\Phi$, as above.  If $w$ is an element of $W$ whose minimum height-cost (relative to $\Phi$) is $\lS(w)$, then for every consecutive substring of the form $s \ s' \ s$ in every reduced expression for $w$, either the edge between $s$ and $s'$ in the Coxeter--Dynkin diagram has length $\infty$, or the length $c_{s'}$ of the root $\lambda_{s'}$ is larger than the length $c_s$ of the root $\lambda_s$.  In particular, $w$ is fully commutative.
\end{proposition}
\begin{proof}
Assume that $(W, S)$ is crystallographic, $w \in W$, and $s, s'$ are two simple reflections such that $w$ has a reduced expression of the form $\cdots s s' s \cdots$.  The product $s s' s$ is a reflection, and the root $\lambda_{s s' s}$ orthogonal to its fixed plane is exactly $s(\lambda_{s'}) = \lambda_{s'} + d_{s, s'} \lambda_s$, of height $1 + d_{s, s'}$.  Since $s s' s$ appears in a reduced expression, $s$ and $s'$ cannot commute, so $m_{s, s'} > 2$.  If $d_{s, s'} = 1$, then the cost of the single reflection $s s' s$ is $\$ 2$, which is strictly less than the total cost when taking $s$, $s'$, $s$ separately; and therefore the cost of $w$ is strictly less than $\lS(w)$ in this case (because we can take the special reduced expression $\cdots s s' s \cdots$ and get each factor for $\$1$ but get the three consecutive factors $ss's$ together for $\$2$).  Taking the contrapositive, if $\height(w) = \lS(w)$ then $d_{s, s'} > 1$.  Consulting \eqref{eq:d}, we see that in this case either $m_{s, s'} = \infty$ or $c_{s'} > c_s$, as claimed.

Now suppose that $w$ satisfies the conditions of the proposition; we wish to show it is fully commutative.  If $w$ is not fully commutative, then it has a reduced word with a consecutive substring having one of the following forms:
\begin{enumerate}
    \item $ss's$ where $m_{s,s'} = 3$, \label{m=3} 
    \item $ss'ss'$ where $m_{s,s'} = 4$, or \label{m=4}
    \item $ss'ss'ss'$ where $m_{s,s'} = 6$. \label{m=6}
\end{enumerate}
When $w$ satisfies the conditions of the proposition, it has no substring of the form listed in \eqref{m=3} because being crystallographic means that the root-lengths $c_s$ and $c_{s'}$ are equal when $m_{s, s'} = 3$. Similarly, the conditions of the proposition prevent \eqref{m=4} and \eqref{m=6} from occurring because then either $ss's$ or $s'ss'$ would violate the requirement that the longer root is always indexed by the central element. Therefore, $w$ is fully commutative.
\end{proof}

\begin{remark}\label{rem:fan remark}
    The reduced word condition in Proposition~\ref{prop:general Coxeter} is very close to the definition of \emph{commutative element} introduced by Fan in \cite[Chapter 11]{FanThesis}, in the context of Weyl groups (i.e., \emph{finite} crystallographic Coxeter groups).  Indeed, they differ in two ways:
    (1) Fan does not consider edges labeled $\infty$ (because he considers only finite groups), and (2) the inequality in edge-lengths in Fan's definition is reversed relative to our statement.  This latter difference is insubstantial: one may replace the root lattice $\Phi$ with its dual lattice
    \[
    \Phi^\vee := \{ \beta \in V : (\beta \mid \lambda) \in \ZZ \text{ for all } \lambda \in \Phi\}
    \]
    (the \emph{coroot lattice}) to reverse all length inequalities.  Thus, for example, the elements in $S^B_n$ that satisfy $\dis(w)/2 = \lS(w)$ are Fan's commutative elements of type C (rather than the commutative elements of type B).   
    
    Fan also asserts (without proof) that in $S^B_n$, using the type C root system, these commutative elements are the ones whose images under the natural embedding $S^B_n \hookrightarrow S_{2n}$ avoid $321$.  This is a close relative of our Proposition~\ref{prop:Stembridge} for $S^B_n$.
\end{remark}

\begin{question}
    Is the converse of Proposition~\ref{prop:general Coxeter} true?
\end{question}

We end this section by observing that Proposition~\ref{prop:general Coxeter} is a strict generalization of Theorem~\ref{thm: equality implies reduced word}.  Indeed, it is not difficult to verify (for example, by induction on the depth) that, in the standard root system of $\affS_n$, one has for each reflection $t$ that $\dis(t)/2 = \height(\alpha_t)$ for every reflection $t$, where $\alpha_t = \lambda_t$ is the positive root perpendicular to the fixed plane of $t$.

The situation in $\affS^C_n$ is similar but slightly more complicated: there are three nonisomorphic crystallographic root systems associated with $\affS^C_n$, which arise by choosing independently whether the simple roots $\lambda_0$ and $\lambda_n$ (corresponding to the simple reflections $s_0$ and $s_n$) should have lengths $\sqrt{2}$ or $\frac{1}{\sqrt{2}}$ times as large as the simple roots $\lambda_1, \ldots, \lambda_{n - 1}$ (which all have the same length).  Of these, the one we want is the system, denoted $C^\vee_n$ in \cite{Macdonald}, in which $\lambda_0$ and $\lambda_n$ are both short, so that
\[
\begin{split}
& s_0(\lambda_1)   = 2\lambda_0 + \lambda_1, \qquad
s_n(\lambda_{n - 1})  = \lambda_{n - 1} + 2\lambda_n, \qquad
\text{ and } \\
& s_i(\lambda_{i \pm 1}) = \lambda_{i \pm 1} + \lambda_i 
\qquad \text{ for } 1 \leq i \leq n - 1.
\end{split}
\]
One can check that with this choice of root system, we again have $\dis(t)/2 = \height(\lambda_t)$ for each reflection $t$.

\begin{question}
    Is there a nice combinatorial formula for the height of elements of $S^B_n$ with respect to the type C root system?  What about for elements of $\affS^C_n$ with respect to the other crystallographic root systems for this group?
\end{question}

\section*{Acknowledgements}

We are grateful to Nathan Reading for pointing out the relationship between our work and root heights, as discussed in Section~\ref{sec:root height}.

\end{document}